\documentclass[10pt,english]{amsart}
\usepackage{amsmath,amssymb,amsthm,amsfonts,graphicx,color}
\usepackage{amssymb}
\usepackage{pdfsync}
\usepackage{bbm}
\usepackage{physics}

\usepackage{mathtools}
\usepackage{amsthm}
\usepackage{amsfonts}
\usepackage{amssymb}
\usepackage{amsmath,graphicx,color}
\usepackage{geometry}
\usepackage{bbm}
\usepackage{enumitem}
\usepackage{physics}
\usepackage{mathrsfs}
\usepackage{epigraph}

\usepackage[colorlinks=true]{hyperref}

\makeatletter
\def\@currentlabel{2.1}\label{e:dispaa}
\def\@currentlabel{2.21}\label{e:dispau}
\def\@currentlabel{2.22}\label{e:dispav}
\def\@currentlabel{2.23}\label{e:dispaw}
\def\@currentlabel{2.24}\label{e:dispax}
\def\theequation{\thesection.\@arabic\c@equation}
\makeatother

\hypersetup{linkcolor=black,citecolor=black,filecolor=black,urlcolor=black} 

\definecolor{dullmagenta}{rgb}{0.4,0,0.4}   
\definecolor{darkblue}{rgb}{0,0,0.4}

\newcommand{\R} {\mathbb R}
\newcommand{\cuad}{{\sqcap\kern-.68em\sqcup}}

\newcommand{\ve}{\varepsilon}

\newcommand{\be}{\begin{equation}}
\newcommand{\ee}{\end{equation}}

\newcommand\ud{\, \textnormal{d}}
\newcommand{\eps}{\epsilon}

\DeclareMathOperator{\arcsinh}{arcsinh}

\renewcommand{\theequation}{\thesection.\arabic{equation}}
 
 \newtheorem{lemma}{Lemma}[section]

\newtheorem{theorem}{Theorem}[section]
\newtheorem{proposition}{Proposition}[section]

\newtheorem{remark}{Remark}[section]
\newcommand{\bremark}{\begin{remark} \em}
\newcommand{\eremark}{\end{remark} }

\begin{document}

\title[Yang-Mills instanton  on a $4+1$ dimensional wormhole]
{Yang-Mills instanton on a four dimensional wormhole: asymptotic stability in the energy space}

\author{Micha{\l} Kowalczyk}
\address{\noindent  M. Kowalczyk - Departamento de
Ingenier\'{\i}a  Matem\'atica and CMM, Universidad de Chile,
Casilla 170 Correo 3, Santiago,
Chile.}
\email{kowalczy@dim.uchile.cl}

\thanks{The first author was partially funded by Chilean research grants FONDECYT 1250156 and ANID project
FB210005. The second author was supported by ANID doctoral fund 21242106.}

\author{Javier Monreal}
\address{\noindent  J. Monreal - Departamento de
Ingenier\'{\i}a  Matem\'atica and CMM, Universidad de Chile,
Casilla 170 Correo 3, Santiago,
Chile.}

\begin{abstract}
In this paper we consider an  $SU(2)$  Yang-Mills field propagating in the  $4+1$ dimensional wormhole spacetime. Assuming the spherically symmetric magnetic ansatz the problem reduces to a one dimensional non linear wave equation. This equation posses a degree one  solution (instanton) which is odd in space. We consider small, odd perturbations of the instanton  and show that it is conditionally asymptotically stable in the odd energy space. 
\end{abstract}
\date{}\maketitle

\section{Introduction}
\setcounter{equation}{0}

\subsection{Yang-Mills field on the four dimensional wormhole}

We  consider an  $SU(2)$  Yang-Mills field propagating in the  $4+1$ dimensional wormhole spacetime $\mathcal M$ with the metric 
\begin{equation}
\label{wormhole 2}
ds^2=-dt^2+dr^2+(r^2+a^2)\left(d\psi^2+\sin^2\psi(d\theta^2+\sin^2\theta d\phi^2)\right).
\end{equation}
In the following we will take $a=1$.
For the $SU(2)$-valued  gauge potential $A_\mu= A_\mu^a\tau_a$ (where the generators $\tau_a$  satisfy the commutation relations $\left[\tau_a, \tau_b\right] = \varepsilon_{abc}\tau_c$) and the curvature $F_{\mu\nu} =\nabla_\mu A_\nu-\nabla_\nu A_\mu +\left[ A_\mu, A_\nu\right]$ the Yang-Mills action is 
\[
S=\int_{\mathcal M} \mathrm{tr}(F_{\alpha\beta} F^{\alpha\beta})
\]
and the corresponding Euler-Lagrange equations are
\begin{equation}
\label{euler-lagrange}
\frac{1}{\sqrt{-g}} \partial_\alpha(\sqrt{-g} F^{\alpha\beta})+\left[ A_\alpha, F^{\alpha\beta}\right]=0.
\end{equation}
We assume the spherically symmetric magnetic ansatz (see \cite{fm} for the derivation)
\[
A=\tau_3(u d\psi+\cos\theta d\phi)+ u\sin\psi(\tau_1 d\theta-\tau_2\sin\theta d\phi)-\cos\psi(\tau_2 d\theta+\tau_1\sin\theta d\phi),
\]
where $u=u(t,r)$. 
For this ansatz the Yang-Mills equations (\ref{euler-lagrange}) reduce to the single semilinear wave equation
\begin{equation}
\label{eq:1}
u_{tt}=u_{rr}+\frac{r}{1+r^2} u_r+\frac{f(u)}{1+r^2}, \quad r\in \R,
\end{equation}
where $f(u)=2 u(1-u^2)$.  

We change variables $x=\arcsinh r$. 
The Yang-Mills equation takes the form  
\begin{equation}
\label{eq:ym in x}
\cosh^2x\, u_{tt}=u_{xx}+f(u).
\end{equation}
The associated conserved energy reads
\[
E(u)=\frac{1}{2}\int \cosh^2 x (\partial_tu)^2+(\partial_x u)^2+ (1-u^2)^2\,\ud x.
\]
Finite energy requires that at both asymptotic ends the fields tend to the vacua:  $u = \pm 1$. In the topologically trivial sector, where $u(\infty) = u(-\infty)$, the energy attains the minimum $E = 0$ at the vacua. For the degree-one solutions, where  $u(\infty) = 1 = -u(-\infty)$, we have the Bogomolnyi inequality
\[
E(u)\geq \frac{1}{2} \int \left(\partial_x u-(1-u^2)\right)^2\,\ud x+\frac{4}{3}.
\]
We note  that the equation
\[
u_{xx}+f(u)=0,
\]
possesses a kink solution $H(x)=\tanh x$ for which $E(H)=\frac{4}{3}$. More generally for any $c\in \R$ the translated kink $H_c(x)=H(x-c)$ is also a solution and $E(H_c)=\frac{4}{3}$. We denote the one-dimensional kink orbit by
\[
\mathcal O_H=\{H_c\mid c\in \R\}.
\]
We put forward the following conjecture regarding the evolution equation (\ref{eq:ym in x}) (stated somewhat informally):

{\bf{Conjecture.}} {\it The orbit  $\mathcal O_H$ is conditionally asymptotically stable in the energy space.}

This conjecture means, roughly speaking, that in general two scenarios are possible for any initial data of the form $u(0)= H_{c_0}+v(0)$:
\begin{enumerate}
\item Later in time $c(t)\to c_\infty\in \R$ and 
\begin{equation}
\label{energy limit}
\lim_{t\to \infty}\int_{-R+c(t)}^{R+c(t)} [v^2(t,r)+v^2_r(t,r)+v^2_t(t,r)]\, dr=0
\end{equation}
for any fixed positive $R$. 
\item The limiting speed is infinite and still we have (\ref{energy limit}). 
\end{enumerate}

In this paper we make the first step towards proving this conjecture assuming that initially $c_0=0$ and that the initial data $(v(0), v_t(0))$ is an odd function of $r\in \R$. This implies immediately that the kink does not move and  $c_\infty=0$. 
Our approach in showing (\ref{energy limit}) follows that of \cite{KMM4, KMMV, KM1} and is based on virial identities. The novelty here is the fact that for the Schr\"odinger operator associated with the linearization around the soliton in variable $r$ is only algebraically decaying and the estimates coming form the virial identities have to be suitably adopted. One of our goals  is to show in this paper the flexibility and wide possible range of application of the virial method.  

One of the most interesting open problems in the theory of nonlinear dispersive equations is the soliton resolution conjecture which asserts, roughly speaking, that generic solutions resolve asymptotically into a coherent structure (typically a collection of solitons) and outgoing radiation.
It was  pointed out in \cite{bk} that the wormhole geometry provides an attractive setting for studying this conjecture. 
In \cite{bk} the soliton resolution conjecture was formulated and verified numerically for equivariant wave maps from the wormhole  into the $3$-sphere. In addition, the rate of convergence to the soliton (which in this case is a harmonic map from a $t = const$ hypersurface of the wormhole into the $3$-sphere) was computed by perturbation methods. Subsequently, the conjecture made in \cite{bk}  was proved  by Rodriguez \cite{rodr 1, rodr 2} via the concentration-compactness method as in \cite{dkm}. For the radially symmetric harmonic maps the solition resolution conjecture in any odd space dimension was proven in \cite{dkm 1}. For the $(2+1)$-dimensional case  the asymptotic behavior of finite-energy solutions as a superposition of asymptotically decoupled harmonic maps and free radiation was described in \cite{Jendrej2+1}.
Regarding studies on wormhole geometries, for problems of similar structure, numerical and rigorous investigations have shown the existence of static chains of alternating kink--antikink solutions that serve as asymptotically stable configurations for the equivariant wave maps from $(2+1)$-dimensional wormhole to $2$ sphere \cite{Jendrej1} (c.f. \cite{{gong_jendrej}, {jkl},  {Jendrej1+1},{JendLaw}}).

Acknowledgement: We would like to thank Piotr Bizo\'n for suggesting to us the problem and for several enlightening discussion during the preparation of the manuscript.  
\section{Statements of the main result}
Considering the previous discussion, we have the following problem
\begin{equation}
\label{eq: in r}
u_{tt}=u_{rr}+\frac{r}{1+r^2} u_r + \frac{f(u)}{1+r^2}, \quad r\in \R, \quad f(u)=2(1-u^2) u.
\end{equation}
We recall that there is a stationary solution of this equation (the kink) 
\[
\begin{aligned}
H(r)&=\tanh(\arcsinh(r))=\frac{r}{\sqrt{1+r^2}},\\
H' (r)&=\frac{1}{(1+r^2)^{3/2}}.
\end{aligned}
\]
We set $u_1=u$, $u_2=u_t$. Expressed in variable $x=\arcsinh r$ 
the energy  is 
\[
E(u_1, u_2)=\frac{1}{2} \int \cosh^2 x\, u_2^2+u_{1,x}^2+W(u_1), 
\quad W(u_1)= (1-u_1^2)^2.
\]
Now let us set $u_1=H+v_1$, $u_2=v_2$ with $v_1, v_2$ odd functions. Then 
\[
\begin{aligned}
E(u_1, u_2)&=\frac{1}{2} \int \cosh^2 x v_2^2+\frac{1}{2}\int (H_{x}+v_{1,x})^2+ W(H+v_1)\\
&=\frac{1}{2} \int \cosh^2 x v_2^2+\frac{1}{2}\int H_{x}^2+W(H)+\int H_{x} v_{1,x} +\frac{1}{2} W'(H) v_1\\
&\quad +\frac{1}{2} \int v_{1,x}^2+\frac{1}{2} W''(H) v_1^2+\frac{1}{2}\int  \left[W(H+v_1)-W(H)-W'(H) v_1-\frac{1}{2}W''(H) v_1^2\right].
\end{aligned}
\]
Denote
\[
B(v_1)=\int v_{1,x}^2+\frac{1}{2} W''(H) v_1^2.
\]
Note that assuming $v_1$ odd we have
\begin{equation}
\label{cond: orto}
\int v_1 H'\,\ud x=0
\end{equation}
and then there exists a constant $\mu>0$ such that 
\[
B(v_1)\geq \mu \|v_1\|_{H^1}^2.
\]
From this it is easy to show
\begin{lemma}\label{lem: stab}
Under the orthogonality condition (\ref{cond: orto}) and supposing  that the initial condition satisfies
\[
\|v_1(0)\|_{H^1}^2+\int \cosh^2x\, v^2_2(0)<\delta
\]
we have 
\[
\|v_1(t)\|_{H^1}^2+\int \cosh^2x\, v^2_2(t)\lesssim \delta
\]
\end{lemma}
Going back to the variable $r= \sinh x$, by Lemma \ref{lem: stab} we have
\[
\int \left[v_{1,r}^2(t)+v_2^2(t)+\frac{v_1^2(t)}{1+r^2}\right] (1+r^2)^{1/2}\,\ud r\lesssim \delta,
\]
for any $t>0$ provided that 
\[
\int \left[v_2^2(0)+v_{1,r}^2(0)+\frac{v_1^2(0)}{1+r^2}\right] (1+r^2)^{1/2}\,\ud r\leq \delta.
\]
%

\subsection{Odd perturbations of a kink for the Yang-Mills  in the energy space}

We consider the following problem
\begin{equation}
\label{eq: v in r}
\begin{aligned}
\dot v_1&=v_2,\\
\dot v_2&= v_{1,rr}+\frac{r}{1+r^2} v_{1,r}+\frac{f'(H)}{1+r^2}+\frac{N(v_1)}{1+r^2},
\end{aligned}
\end{equation}
where
\[
N(v_1)=f(H+v_1)-f(H)-f'(H) v_1=-6 H v_1^2- v_1^3,
\]
and we assume that  $v_j$ are odd functions of $r$.

Our main result is 
\begin{theorem}\label{thm:one}
There exists $\delta>0$ such that if the initial data $(v_1(0), v_2(0))$ are odd functions satisfying 
\[
\int \left(v^2_2(0)+v^2_{1,r}(0)+\frac{v^2_1(0)}{1+r^2}\right)(1+r^2)^{1/2}\, d r<\delta,
\]
then for any $R>0$   
\[
\lim_{t\to \infty} \int_{-R}^R \left(v^2_2(t)+v^2_{1,r}(t)+v^2_1(t)\right)\, dr=0,
\]
and 
\[
\lim_{t\to \infty} \int \frac{\min\{1, r^{-1}\}}{(1+r^2)^{1/2}} v^2_1(t)=0.
\]
\end{theorem}
In order to prove our result we further reformulate the problem. Denote
\[
L v=-v_{rr}-\frac{r}{1+r^2} v_{r}-\frac{f'(H)}{1+r^2} v,
\]
and define
\[
\tilde L \phi=(1+r^2)^{1/4}L\left(\frac{\phi}{(1+r^2)^{1/4}}\right).
\]
Setting
\[
w_j= (1+r^2)^{1/4}v_j, \quad j=1,2, 
\]
we get
\begin{equation}
\label{eq: w}
\begin{aligned}
\dot w_1&=w_2,\\
\dot w_2&=-\tilde L w_1+N(w_1).
\end{aligned}
\end{equation}
Explicitly
\[
\tilde L=-\partial_{rr} + V,
\]
with 
\begin{equation}
\label{eq:V0}
V=\frac{2-r^2}{4(1+r^2)^2}-\frac{f'(H)}{1+r^2}=\frac{3}{4}\frac{5r^2-2}{(1+r^2)^2},
\end{equation}
and 
\[
N(w_1)=(1+r^2)^{-3/4}N\left(\frac{w_1}{(1+r^2)^{1/4}}\right)=-\frac{6 H w_1^2}{(1+r^2)^{5/4}}-\frac{2w_1^3}{(1+r^2)^{3/2}}.
\]

Theorem \ref{thm:one} is a direct corollary from: 
\begin{theorem}\label{thm:2}
Suppose that the initial data $(w_1(0,r), w_2(0,r))$ of (\ref{eq: w}) are odd functions of $r$. There exists $\delta>0$ such that if
\[
\norm{\frac{w_1(0)}{(1+r^2)^{1/2}}}_{L^2}+\norm{{w_{1,r}}(0)}_{L^2}+\norm{w_2(0)}_{L^2} <\delta,
\] 
then for any $R>0$ 
\[
\lim_{t\to \infty}\left(\int_{-R}^R |w_2(t)|^2+|\partial_r w_1(t)|^2 + |w_1(t)|^2\right)=0,
\]
and also 
\[
\lim_{t\to \infty} \int \frac{{\min\{1,r^{-1}\}}}{1+r^2} |w_1(t)|^2=0.
\]
\end{theorem}

The heart of our method is the use of the Darboux transformation as in \cite{KMM4, KM1} to transform (\ref{eq: w}) into a problem whose potential is "repulsive". We will check this condition now for the problem at hand.  We set 
\[
\tilde Y=(1+r^2)^{1/4}  H'.
\]
We know 
\[
\tilde L \tilde Y=0.
\]
We define the operators
\[
U= \tilde Y\cdot\partial_r\cdot  \tilde Y^{-1}, \quad U^\ast= -\tilde Y^{-1}\cdot \partial_r\cdot \tilde Y.
\]
Then
\[
\tilde L= U^\ast U,
\]
and if we let
\[
\tilde L_1=U U^\ast,
\]
then
\[
U\tilde L=\tilde L_1 U.
\]
Explicitly
\[
\tilde L_1=-\partial_{rr}-\tilde Y \partial_r\left(\tilde Y^{-2} \partial_r \tilde Y\right).
\]
We see that applying $U$ to the equation leads possibly to a better problem, meaning that $\tilde L_1$ does not have a kernel and is a positive operator. Indeed, 
we have $H'(r)=(1+r^2)^{-3/2}$. As a consequence $\tilde Y(r)=(1+r^2)^{-3/4}$. Then
\[
P_1=-\tilde Y \partial_r\left(\tilde Y^{-2} \partial_r \tilde Y\right)=\frac{3}{4}\frac{2+r^2}{(1+r^2)^2}.
\]
As in  \cite{KM1} we get the following criterion for asymptotic stability
\[
-rP_1'\geq 0.
\]
We check directly
\begin{equation}
\label{cond:repuls}
-rP_1'=\frac{3}{2} r^2\frac{3+r^2}{(1+r^2)^3}\geq 0.
\end{equation}
As we have mentioned above the fact that the potential $P_1$ is only algebraically decaying (as opposed to exponential decay) introduces new elements  in the  application of  our method. 

\section{Preliminaries}

\subsection{General  virial calculation}

Let $\Phi:\R\to\R$ be a smooth, odd, strictly increasing, bounded function and let $\boldsymbol a=(a_1, a_2)\in H^1\times L^2$ be a  solution of 
\[
\begin{cases}
\dot a_1= a_2 \\
\dot a_2=-L a_1+G
\end{cases}
\] 
where $L=-\partial_{rr}+ P
$
and $G=G(t,r)$, $P=P(r)$ are given functions.
Define
\begin{equation*}
\mathcal A=\int\left(\Phi \partial_r a_1+\frac{1}{2}\Phi' a_1\right)a_2 \,.
\end{equation*}
We compute formally
\begin{equation}\label{eq:vir0}
\dot{\mathcal A}=-\int\Phi' (\partial_r a_1)^2+ \frac{1}{4}\int \Phi''' a_1^2 +\frac{1}{2} \int \Phi P'a_1^2+\int \left(\Phi \partial_r F+\frac{1}{2}\Phi' F\right) a_2
+\int G\left(\Phi \partial_r a_1+\frac{1}{2}\Phi' a_1\right).
\end{equation}
Changing variables
\[
\tilde a_1=\zeta a_1, \qquad \zeta=\sqrt{\Phi'}\, ,
\]
we recall the alternative identity
\begin{equation}
\label{eq:vir1}
\begin{aligned}
\dot{\mathcal A} &=-\int (\partial_r \tilde a_1)^2-\frac{1}{2} \int \left(\frac{\zeta''}{\zeta} - \frac{(\zeta')^2}{\zeta^2}\right) \tilde a_1^2
+\frac 12 \int \frac{\Phi P'}{\zeta^2} \tilde a_1^2+\int G\left(\Phi \partial_r a_1+\frac{1}{2}\Phi' a_1\right).
\end{aligned}
\end{equation}
We refer to \cite[Proof of Proposition 1]{KMM4} for detailed computations.

\subsection{Weight functions for  virial identities}

We consider a smooth even function $\chi:\R\to \R$ satisfying
\begin{equation*}
\mbox{$\chi=1$ on $[-1,1]$, $\chi=0$ on $(-\infty,-2]\cup[2,+\infty)$,
$\chi'\leq 0$ on $[0,+\infty)$.}
\end{equation*}
The constants $A$ and $B$ will be fixed later such that
$1\ll B \ll A$.
We define the following functions on $\R$
\begin{equation*}
\zeta_A(r)=\exp\left(-\frac{1}{A}\left(1-\chi(r)\right)|r|\right), \quad \Phi_A(r)=\int_0^r \zeta_A^2(y)\,\ud y,
\end{equation*}
and
\begin{align}
&\zeta_B(r)=\exp\left(-\frac{1}{B}\left(1-\chi(r)\right)|r|\right), \quad \Phi_B(r)=\int_0^r \zeta_B^2(y)\,\ud y,\nonumber\\
&\Psi_{A,B}(r)=\chi_A^2(r)\Phi_B(r)\quad\mbox{where}\quad
\chi_A(r)=\chi\left(\frac{r}{A}\right).\label{def:chiB}
\end{align}
We also set
\[
\sigma_A(r) = \sech\left( \frac rA \right).
\]
Note that  $\sigma^2_A\lesssim \zeta_A^2\lesssim \sigma^2_A$. 
Finally, we let 
\[
\rho_K(r)=\sech\left(\frac{r}{K}\right)
\]
with some $K>0$ to be fixed later. 

\section{Virial estimates at large scale}
\setcounter{equation}{0}

\subsection{First virial estimate}

Our goal is to show the following
\begin{proposition}
For odd perturbations $(w_1, w_2)$ and any $T>0$ we have:
\[
\int_0^T \| \sigma_A \partial_r w_1\|_{L^2}^2+\frac{1}{A^2}\|\sigma_A w_1\|_{L^2}^2+\frac{1}{A^2} \| \sigma_A w_2\|_{L^2}^2 \lesssim A\delta^2+ \int_0^T \| \rho^2_K w_1\|_{L^2}^2.
\]
\end{proposition}
The rest of this section is devoted to the proof of this Proposition. 

We consider equation (\ref{eq: w}) and define
\[
\mathcal I=\int\left(\Phi_A \partial_r w_1+\frac{1}{2}\Phi'_A w_1\right)w_2.
\]
From (\ref{eq:vir0}) it follows 
\[
\begin{aligned}
\dot{\mathcal I} & =-\int\Phi_A' (\partial_r w_1)^2+ \frac{1}{4}\int \Phi_A''' w_1^2 +\frac{1}{2} \int \Phi_A V' w_1^2\\
& \quad +\int N(w_1)\left(\Phi_A \partial_r w_1+\frac{1}{2}\Phi_A' w_1\right).
\end{aligned}
\]
Changing variables
\[
\tilde w_1=\zeta_A w_1, \qquad \zeta_A=\sqrt{\Phi_A'} ,
\]
we get the alternative identity (c.f. (\ref{eq:vir1})):
\begin{equation}
\label{eq:vir2}
\begin{aligned}
\dot{\mathcal I} &=-\int (\partial_r \tilde w_1)^2-\frac{1}{2} \int \left(\frac{\zeta_A''}{\zeta_A} - \frac{(\zeta_A')^2}{\zeta_A^2}\right) \tilde w_1^2
+\frac 12 \int \frac{\Phi_A V'}{\zeta_A^2} \tilde w_1^2\\
&\quad +\int N(w_1)\left(\Phi_A\partial_r w_1+\frac{1}{2}\Phi_A' w_1\right).
\end{aligned}
\end{equation}
We begin by estimating the first line above. 
We have
\begin{equation}
\label{CotaIdentidad}
\left|\left(\frac{\zeta_A''}{\zeta_A} - \frac{(\zeta_A')^2}{\zeta_A^2}\right)\right|=\left|\left(\frac{\zeta_A'}{\zeta_A}\right)'\right|\lesssim  \frac{1}{A^{2}}\mathbbm{1}_{[-2,2]}(r).
\end{equation}
Explicitly
\[
V(r)=\frac{3}{4}\frac{5r^2-2}{(1+r^2)^2}
\]
and
\[
\begin{aligned}
V'(r)=-\frac{15}{2}\frac{r^3}{(1+r^2)^3}+\frac{27}{2} \frac{r}{(1+r^2)^3}.
\end{aligned}
\]
Let $K>\sqrt{9/5}$. We can write 
\begin{equation}
\label{eq:choice k}
\begin{aligned}
-\Phi_A V'=-\frac{\Phi_A}{r} r V'= \frac{\Phi_A}{r} \frac{15r^2}{2(1+r^2)^3}\left(r^2-\frac{9}{5}\right)&= \frac{\Phi_A}{r} \frac{15r^2}{2(1+r^2)^3}\left(r^2-\frac{9}{5}\right)\mathbbm{1}_{[-K,K]}\\
&\quad +\frac{\Phi_A}{r} \frac{15r^2}{2(1+r^2)^3}\left(r^2-\frac{9}{5}\right)\mathbbm{1}_{[-K,K]^c}.
\end{aligned}
\end{equation}
Then, with some positive constants $C_1$, $C_2$: 
\[
-\Phi_A V'\geq -C_1 \mathbbm{1}_{[-K,K]} + C_2 \frac{\Phi_A}{r} \frac{1}{1+r^2} \mathbbm{1}_{[-K,K]^c}.
\]
Note that the first member above is compactly supported and the second member is decaying and positive. 
We get
\[
-\frac 12 \int \frac{\Phi_A V'}{\zeta_A^2} \tilde w_1^2\geq -C_1 \int w_1^2 \mathbbm{1}_{[-K,K]}+C_2\int w_1^2 \frac{\Phi_A}{r} \frac{1}{1+r^2} \mathbbm{1}_{[-K,K]^c}.
\]

To estimate the nonlinear terms we need some simple intermediate results. Consider the following lemma that relates orbital stability of $v=(v_1,v_2)$ for the weighted norms associated to $w=(w_1,w_2)$:
\begin{lemma}
The orbital stability for $v$ translates as:
\[
\norm{\frac{w_1(t)}{(1+r^2)^{1/2}}}_{L^2}+\norm{{w_{1,r}(t)}}_{L^2}+\norm{w_2(t)}_{L^2} \lesssim \delta
\]
for $t\geq 0$.
\begin{proof}
From the orbital stability we have
\[
\int v_1^2+v_{1,x}^2 +\cosh^2 x v_2 dx\lesssim \delta^2,
\]
then, by considering the change of variables $r=\sinh x$, $w_j=(1+r^2)^{1/4}v_j$ we get
\begin{equation}
\label{worbitalproof}
\int \frac{w_1^2}{(1+r^2)}dr+\int \left[\left(\frac{w_1}{(1+r^2)^{1/4}} \right)_r\right]^2 (1+r^2)^{1/2}dr+ \int w_2^2 dr \lesssim \delta^2.
\end{equation}
Expanding the middle term we obtain
\[
\int \left[\left(\frac{w_1}{(1+r^2)^{1/4}} \right)_r\right]^2 (1+r^2)^{1/2}dr=\int w_{1,r}^2+\frac14 \frac{r^2}{(1+r^2)^2}w^2-\frac{r}{1+r^2}w_r w dr.
\]
Replacing the expansion in \eqref{worbitalproof} and reorganizing we obtain
\begin{align*}
\int\frac{w_1^2}{1+r^2}+w_{1,r}^2+\frac{1}{4}\frac{r^2}{(1+r^2)^2}w_1^2+w_2^2 dr &\leq C\delta^2+\int \frac{|r|}{1+r^2}|w_{1,r}| |w_1|\\
&\leq C\delta^2+\int \frac{|w_{1,r}| |w_1|}{(1+r^2)^{1/2}}\\
&\leq C\delta^2 +\left(\int \frac{w_1^2}{1+r^2} \right)^{1/2} \left(\int w_{1,r}^2 \right)^{1/2}\\
&\leq C\delta^2+ \frac{1}{2} \int \frac{w_1^2}{1+r^2}+\frac12 \int w_{1,r}^2,
\end{align*}
from where we conclude
\[
\int \frac{w_1^2}{1+r^2}+w_{1,r}^2+w_2^2\lesssim \delta^2.
\]
\end{proof}
\end{lemma}

We have 
\begin{lemma}
The following estimate holds
\[
\norm{\frac{w_1}{(1+r^2)^{1/4}}}_{L^\infty} \lesssim \delta.
\]
\begin{proof}
Notice that
\begin{align*}
\int_0^{\bar r} \frac{w_1 w_{1,r}}{(1+r^2)^{1/2}}&=\frac{1}{2} \int_0^r \frac{(w_1^2)_r}{(1+r^2)^{1/2}}\\
&=\frac{1}{2}\frac{w_1^2({\bar r})}{(1+{\bar r}^2)^{1/2}}+\frac{1}{2}\int_0^r \frac{r w_1^2}{(1+r^2)^{3/2}}.
\end{align*}
Then, by the Cauchy-Schwarz inequality  we get
\[
\frac{w_1({\bar r})}{(1+{\bar r}^2)^{1/4}} \lesssim \norm{\frac{w_1}{(1+r^2)^{1/2}}}_{L^2}\cdot \norm{w_{1,r}}_{L^2}+\norm{\frac{w_1}{(1+r^2)^{1/2}}}_{L^2}.
\]
We conclude\[
\norm{\frac{w_1(r)}{(1+r^2)^{1/4}}}_{L^\infty}\lesssim \delta.
\]
\end{proof}
\end{lemma}

With the previous estimate we are able to get control over the nonlinear terms in (\ref{eq:vir2}). 
Explicitly we have
\[
(1+r^2)^{1/4}N\left(\frac{w_1}{(1+r^2)^{1/4}}, r\right)=-\frac{6 H w_1^2}{(1+r^2)^{5/4}}-\frac{2 w_1^3}{(1+r^2)^{3/2}}.
\]
It follows
\[
\begin{aligned}
&-\int (1+r^2)^{1/4}N\left(\frac{w_1}{(1+r^2)^{1/4}}, r\right)\left(\Phi_A\partial_r w_1+\frac{1}{2}\Phi_A' w_1\right)= -2\int w_1^3\Phi_A \partial_r\left(\frac{H}{(1+r^2)^{5/4}}\right)\\
&\quad -\frac{1}{2}\int w_1^4 \Phi_A \partial_r\left(\frac{1}{(1+r^2)^{3/2}}\right)+\int \frac{w_1^3 H}{(1+ r^2)^{5/4}} \Phi_A'+\frac{1}{2}\int\frac{w_1^4}{(1+r^2)^{3/2}}\Phi_A'\\
&\quad 
=K_1+K_2+K_3+K_4.
\end{aligned}
\]
To  estimate $K_1$ from below we note that, 
\[
\Phi_A \partial_r\left(\frac{H}{(1+r^2)^{5/4}}\right)=\frac{\Phi_A}{r} r \partial_r\left(\frac{H}{(1+r^2)^{5/4}}\right)\geq -C_3\mathbbm{1}_{[-K,K]}- C_4 \frac{\Phi_A}{r} \frac{1}{(1+r^2)^{5/4}} \mathbbm{1}_{[-K,K]^c}.
\]
Then 
\[
K_1\geq -C_3 \norm{\frac{w_1}{(1+r^2)^{1/4}}}_{L^\infty} \int w_1^2 \mathbbm{1}_{[-K,K]}- C_4 \norm{\frac{w_1}{(1+r^2)^{1/4}}}_{L^\infty} \int w_1^2  \frac{\Phi_A}{r} \frac{1}{(1+r^2)} \mathbbm{1}_{[-K,K]^c}.
\]
Similarly we estimate $K_2$ 
\[
K_2\geq -C_5 \norm{\frac{w_1}{(1+r^2)^{1/4}}}_{L^\infty}^2 \int w_1^2 \mathbbm{1}_{[-K,K]}- C_6 \norm{\frac{w_1}{(1+r^2)^{1/4}}}^2_{L^\infty} \int w_1^2  \frac{\Phi_A}{r} \frac{1}{(1+r^2)} \mathbbm{1}_{[-K,K]^c}.
\]
Next, we note using Claim 1 in \cite{KMM4} 
\begin{align*}
\left |\int \frac{w_1^3 H}{(1+ r^2)^{5/4}} \Phi_A'\right|&\leq \int \left(\frac{w_1}{(1+r^2)^{1/4}} \right)^3\cdot \frac{\Phi_A'}{(1+r^2)^{1/2}}\\
&\lesssim \int \left|\frac{w_1}{(1+r^2)^{1/4}} \right| ^3\zeta_A^2\\
&\lesssim A^2 \norm{\frac{w_1}{(1+r^2)^{1/4}}}_{L^\infty} \int \left| \partial_r \left( \zeta_A \frac{w_1}{(1+r^2)^{1/4}}\right)\right|^2,
\end{align*}
hence
\[
K_3\gtrsim -A^2 \norm{\frac{w_1}{(1+r^2)^{1/4}}}_{L^\infty} \int \left| \partial_r \left( \zeta_A \frac{w_1}{(1+r^2)^{1/4}}\right)\right|^2,
\]
and similarly 
\[
K_4\gtrsim -A^2 \norm{\frac{w_1}{(1+r^2)^{1/4}}}_{L^\infty}^2 \int \left| \partial_r \left( \zeta_A \frac{w_1}{(1+r^2)^{1/4}}\right)\right|^2.
\]
It is easy to see that 
\[
\begin{aligned}
\int \left| \partial_r \left( \zeta_A \frac{w_1}{(1+r^2)^{1/4}}\right)\right|^2& \lesssim \int (\partial_r \tilde w_1)^2+\int \frac{\tilde w_1^2}{(1+r^2)^{1/2}}\\
&\lesssim \int (\partial_r \tilde w_1)^2+\int \sigma_A^2 w_1^2.
\end{aligned}
\]
Summarizing, provided that  $\delta \gtrsim \norm{\frac{w_1}{(1+r^2)^{1/4}}}_{L^\infty}$, $A^2\delta$  is small, we get
\[
\begin{aligned}
-\dot{\mathcal I}+ C_7\int w_1^2 \mathbbm{1}_{[-K,K]}+ A^2 \delta \int \sigma_A^2 w_1^2 &\gtrsim \int |\partial_r \tilde w_1|^2 +\int w_1^2 \frac{\Phi_A}{r} \frac{1}{1+r^2} \mathbbm{1}_{[-K,K]^c}.
\end{aligned}
\]
As in the proof of (18) in  \cite{KM1} we infer 
\[
\int \sigma^2_A |\partial_r w_1|^2+\frac{1}{A^2}\int \sigma_A^2 w_1^2\lesssim \int |\partial_r \tilde w_1|^2+\frac{1}{A} \int w_1^2 \mathbbm{1}_{[-K,K]},
\]
hence provided $A^2\delta\ll A^{-2}$ 
\begin{equation}
\label{est:virial1}
-\dot{\mathcal I}+ \int w_1^2 \mathbbm{1}_{[-K,K]}\gtrsim \int \sigma^2_A |\partial_r w_1|^2+\frac{1}{A^2}\int \sigma_A^2 w_1^2+\int  \frac{\Phi_A}{r} \frac{1}{1+r^2} w_1^2,
\end{equation}
in particular
\[
-\dot{\mathcal I}+ \int w_1^2 \mathbbm{1}_{[-K,K]}\gtrsim \int \sigma^2_A |\partial_r w_1|^2+\frac{1}{A^2}\int \sigma_A^2 w_1^2,
\]
integrating in time gives
\begin{equation}
\label{est:virial1T}
A \delta^2+\int_0^T \| \rho_K w_1\|_{L^2}^2 \gtrsim \int_0^T \|\sigma_A \partial_r w_1\|_{L^2}^2 +\frac{1}{A^2}\|\sigma_A w_1\|_{L^2}^2.
\end{equation}
\subsection{Second virial estimate}

We define
\[
\mathcal H=\int \sigma^2_A w_1 w_2.
\]
Differentiating in $t$ we get
\[
\dot{\mathcal H}=\int \sigma^2_A w_2^2-\int \sigma^2_A \left(|\partial_r w_1|^2+V_0 w_1^2\right)-\int \sigma^2_A\left(\frac{6 H w_1^3}{(1+r^2)^{5/4}} + \frac{2 w_1^4}{(1+r^2)^{3/2}}\right).
\]
It is not hard to see that 
\begin{equation}
\label{eq:virialtwo}
\int \sigma^2_A w_2^2\lesssim \dot{\mathcal H}+\int \sigma^2_A |\partial_r w_1|^2+\int \sigma^2_A w_1^2.
\end{equation}
Combining (\ref{est:virial1T}) and (\ref{eq:virialtwo}) we get
\begin{equation}
\label{est:virial12}
\int_0^T \|\sigma_A \partial_r w_1\|_{L^2}^2+\frac{1}{A^2}\|\sigma_A w_1\|_{L^2}^2+\frac{1}{A^2}\|\sigma_A w_2\|_{L^2}^2\lesssim A\delta^2+\int_0^T \|\rho^2_K w_1\|_{L^2}^2.
\end{equation}

\section{Auxiliary Lemmas}
\setcounter{equation}{0}
For each small $\eps>0$ we define the operator  $X_\eps$ in terms of it Fourier transform 
\[
\widehat{X_\eps h}(\xi)=\frac{\hat h(\xi)}{1+\eps \xi^2}.
\]
In other words 
\[
X_\eps h= g\Longrightarrow h=(-\eps\partial_{rr}+1) g.
\]
Below we gather some basic properties of $X_\epsilon$.
\begin{lemma}[Lemma 4.7 \cite{KMMV}]\label{lem:xeps}
It holds
\[
\begin{aligned}
\|X_\eps h\|_{L^2}&\leq \|h\|_{L^2},\\
\|\partial_r X_\eps h\|_{L^2}&\leq \eps^{-\frac{1}{2}} \|h\|_{L^2},\\
\|\partial_{rr} X_\eps h\|_{L^2}&\leq \eps^{-1} \|h\|_{L^2}.
\end{aligned}
\]
\end{lemma}

Next we define the operator $S_\eps = X_\eps U$.   
We have the following transfer estimate (c.f. similar results in \cite{KMMV,  KM1}).
\begin{lemma}\label{lem:transfer}
For any $K>0$ and for any odd function $u$ we have
\[
\|\rho_K^2\sigma_B^2 u\|_{L^2}\lesssim  K \|\rho_K \sigma_B S_\eps u\|_{L^2}+\eps  K \| \rho_K\sigma_B \partial_rS_\eps u\|_{L^2}.
\]
\end{lemma}
\begin{proof}
We recall that with $\tilde Y=(1+r^2)^{-\frac{3}{4}}$ we have 
\[
U u=\tilde Y\cdot \partial_r\cdot  \tilde Y^{-1} u.
\]
Set $v=S_\eps u$.  Note that $v$ is even in $r$. 
Then 
\[
X_\eps^{-1} v= U u,
\]
hence
\[
\partial_r( \tilde Y^{-1}u)= \tilde Y^{-1}X_\eps^{-1}  v=-\eps \tilde Y^{-1} \partial_{rr}v + \tilde Y^{-1} v.
\]
Integrating the above expression on $(0,r)$ and using $u(0)=0$, $\partial_r v(0)=0$  we get
\[
\begin{aligned}
\tilde Y^{-1}(r)u(r)&=-\eps \int_0^r \tilde Y^{-1} \partial_{rr}v+\int_0^r \tilde Y^{-1} v\\
&=-\eps  \tilde Y^{-1} (r)\partial_{r}v(r)+\eps \int_0^r \partial_r \tilde Y^{-1} \partial_{r}v+\int_0^r \tilde Y^{-1} v,
\end{aligned}
\] 
hence 
\begin{equation}
\sigma^2_B(r)\rho_K^2(r) u(r)= -\eps \sigma^2_B(r) \rho_K^2(r) \partial_{r}v(r)+\eps \sigma^2_B \rho_K^2 \tilde Y \int_0^r \partial_r \tilde Y^{-1} \partial_{r}v+\sigma^2_B\rho_K^2 \tilde Y\int_0^r \tilde Y^{-1} v.
\label{eq:rhoKu}
\end{equation}
We have
\[
\begin{aligned}
\left| \tilde Y(r)\int_0^r \partial_r \tilde Y^{-1} \partial_{r}v\right|&\lesssim \|\rho_K\sigma_B\partial_r v\|_{L^2}\sigma^{-1}_B(r)\rho_K^{-1}(r), \\
\left|\tilde Y\int_0^r \tilde Y^{-1} v\right|&\lesssim \|\rho_K\sigma_B v\|_{L^2} \sigma^{-1}_B(r)\rho_K^{-1} (r). 
\end{aligned}
\]
The desired estimate follows now easily from (\ref{eq:rhoKu}). 
\end{proof}

Our next result will be used repeatedly in what follows.
\begin{lemma}\label{lem:comut 1}
It holds
\begin{equation}
\label{est:comut 1}
\|g X_\eps f\|_{L^2}\lesssim \|X_\eps(gf)\|_{L^2}+\eps^{1/2} \|\partial_r g X_\eps f\|_{L^2}+\eps\|\partial_{rr} g X_\eps f\|_{L^2}.
\end{equation}
\end{lemma}
\begin{proof}
Denote
\[
X_\eps(gf)=k, \quad X_\eps f= h. 
\]
We can write, 
\begin{align*}
X_\eps^{-1}(gh)=gh-\eps \partial_{rr}(gh)\\
X_\eps^{-1}(k)=gh-\eps g \partial_{rr}h.
\end{align*}
Then we have
\[
X_\eps^{-1}(gh)-X^{-1}_\eps k=-\eps(2\partial_r (\partial_r g h)-\partial_{rr} g h),
\]
to see this notice that
\begin{align*}
X_\eps^{-1}(gh)-X_\ve^{-1}(k)&=-\eps \partial_{rr} (gh)+\eps g\partial_{rr}h\\
&=\eps \left(g \partial_{rr} h-\partial_{rr} gh-g\partial_{rr}h-2\partial_r g \partial_r h \right)\\
&=-\eps \left(\partial_{rr}g h+2\partial_r g \partial_r h \right)\\
&=-\eps \left(2\partial_r(\partial_r g h)-\partial_{rr} gh \right).
\end{align*}
Hence
\[
\|gh\|_{L^2}^2\lesssim \|k\|^2_{L^2}+\eps^2\|\partial_r X_\eps(\partial_r g h)\|_{L^2}^2+\eps^2 \|X_\eps (\partial_{rr} g h)\|_{L^2}^2.
\]
Estimate (\ref{est:comut 1}) follows from this and Lemma \ref{lem:xeps}.
\end{proof}
Next we will show 
\begin{lemma}\label{lem:coerc}
For any sufficiently large $A$ it holds 
\[
\begin{aligned}
\|\sigma_A S_\eps u\|_{L^2}& \lesssim \eps^{-\frac{1}{2}} \|\sigma_A u\|_{L^2},\\
\|\sigma_A \partial_r S_\eps u\|_{L^2}& \lesssim  \eps^{-1} \|\sigma_A u\|_{L^2}.
\end{aligned}
\]
Similarly, for any large $B$ and $K$, $B\gg K$,  it holds
\[
\|\sigma^2_B\rho^2_K S_\eps u\|_{L^2}\lesssim \eps^{-\frac{1}{2}} \|\sigma^2_B\rho^2_K u\|_{L^2}.
\] 
\end{lemma}
\begin{proof}
We will show the first and second estimate only the proof of the third estimate being similar. We have
\[
U=\partial_r+ \nu, \quad \nu(r)=\frac{3}{2} \frac{r}{1+r^2}.
\]
Then 
\[
\sigma_A S_\eps u=\sigma_A X_\eps(\partial_r u)+\sigma_A X_\eps (\nu u).
\]
We have for large $A$
\[
|\partial_r \sigma_A|\lesssim \sigma_A, \quad |\partial_{rr} \sigma_A|\lesssim \sigma_A.
\]
Using Lemma \ref{lem:comut 1} with $g=\sigma_A$ we get
\[
\|\sigma_A S_\eps u\|_{L^2}\lesssim \| X_\eps(\sigma_A \partial_r u)\|_{L^2}+\|X_\eps (\sigma_A \nu u)\|_{L^2},
\]
to see this, notice that for the first term:
\[
\|\sigma_A X_\ve( \partial_r u)\|_{L^2}\lesssim \|X_\ve(\sigma_A \partial_r u)\|_{L^2}+\eps^{1/2}\|\partial_r \sigma_A X_\eps \partial_r u\|_{L^2}+\eps \|\partial_{rr}\sigma_A X_\eps \partial_r u\|_{L^2}.
\]
Now for big enough $A$ we can use $|\partial_r \sigma_A|\lesssim \sigma_A$, and $|\partial_{rr} \sigma_A|\lesssim \sigma_A$. Thus, we obtain
\[
\|\sigma_A X_\eps (\partial_r u)\|_{L^2}(1-\eps^{1/2}-\eps)\lesssim \| X_\eps (\sigma_A \partial_r u)\|_{L^2}.
\]
Then, for $\eps$ such that $\eps^{1/2}+\eps \leq 1/2$ we obtain
\[
\| \sigma_A X_\eps(\partial_r u)\|_{L^2} \lesssim \| X_\eps (\sigma_A \partial_r u)\|_{L^2}.
\]
The estimate for the second term follows by a similar argument.
Since
\[
X_\eps(\sigma_A \partial_r u)=\partial_r X_\eps(\sigma_A u)- X_\eps(\partial_r\sigma_A u),
\]
the first estimate follows now from this and Lemma \ref{lem:xeps}. 

To show the second estimate we use
\[
\begin{aligned}
\sigma_A\partial_r (S_\eps u)& =\sigma_A\left( \partial_{rr} X_\eps u+\partial_r X_\eps (\nu u)\right),
\end{aligned}
\]
Lemma \ref{lem:comut 1} and  Lemma \ref{lem:xeps} again. 
\end{proof}

The following Lemma is a modification of Lemma 4.6 in \cite{KMMV}.
\begin{lemma}\label{lem:4.6}
It holds 
\[
\|\rho_K w\|_{L^2}^2\lesssim K^2\|\partial_r w\|_{L^2}^2+K\int(-\Phi_B P_1')  w^2.
\]
\end{lemma}
\begin{proof}
For any $y\in \R$ we have
\[
2\int_y^\infty \partial_r w w e^{\,-2(r-y)/K}\,\ud r=-w^2(y)+ \int_y^\infty w^2 \partial_r\left(e^{\,-2(r-y)/K}\right)\,\ud r.
\]
In the last integral we have 
\[
 \partial_r\left(e^{\,-2(r-y)/K}\right)=   -\frac{2}{K}e^{\,-2(r-y)/K},
\]  
Rearranging and using the Cauchy-Schwarz inequality we get
\[
\frac{1}{K} \int_y^\infty w^2  e^{\,-2(r-y)/K}\,\ud r\lesssim K \int_y^{\infty}  |\partial_r w|^2\,\ud r+ w^2(y).
\]
Arguing similarly we get
\[
\frac{1}{K} \int^y_{-\infty} w^2  e^{\,-2(r-y)/K}\,\ud r\lesssim K \int^y_{-\infty}  |\partial_r w|^2\,\ud r+  w^2(y).
\]
Adding the above we obtain
\[
\frac{1}{K} \int w^2   e^{\,-2|r-y|/K}\,\ud r\lesssim K \int  |\partial_r w|^2\,\ud r+  w^2(y).
\]
Using
\[
\frac{1}{K} \int w^2  e^{\,-2|r-y|/K}\,\ud r\geq \frac{e^{\,-|y|/K}}{K} \int w^2  e^{\,-|r|/K}\,\ud r,
\]
and multiplying by $-\Phi_B(y) P'_1(y)\geq 0$ (c.f. (\ref{cond:repuls})) we get
\[
\begin{aligned}
(-\Phi_B(y)P'_1(y))\frac{e^{\,-2|y|/K}}{K} \int w^2  e^{\,-2|r|/K}\,\ud r&\lesssim K (-\Phi_B(y) P'(y))\int  |\partial_r w|^2\,\ud r\\
&\quad+ (-\Phi_B(y) P'(y))  w^2(y).
\end{aligned}
\]
Since $(- \Phi_B P'_1)\in L^1(\R)$ integrating the above inequality we conclude. 

\end{proof}

\section{Third virial estimate}
\setcounter{equation}{0}

We will transform (\ref{eq: w}) by setting $u_j=S_\eps w_j$. The transformed problem takes form 
\begin{equation}
\label{eq:u}
\begin{aligned}
\dot u_1&=u_2,\\
\dot u_2&=-\tilde L_1 u_1-\left[X_\eps, P_1\right] U w_1+S_\eps\left[(1+r^2)^{1/4}N\left(\frac{w_1}{(1+r^2)^{1/4}}, r\right)\right].
\end{aligned}
\end{equation}
Define 
\[
\mathcal J=\int\left(\Psi_{A,B} \partial_r u_1+\frac{1}{2}\Psi'_{A,B} u_1\right)u_2.
\]
We take time derivative of $\mathcal J$ and use (\ref{eq:vir0}). Setting 
\[
\tilde u_1= \chi_A\zeta_{B} u_1, 
\]
we get from (\ref{eq:vir1}) and calculations in section 4.3 in \cite{KMM4}
\begin{equation}
\label{eq:vir 2}
\begin{aligned}
\dot{\mathcal J}=&-\int (\partial_r \tilde u_1)^2-\frac{1}{2} \int \left(\frac{\zeta_{B}''}{\zeta_{B}} - \frac{(\zeta_{B}')^2}{\zeta_{B}^2}\right) \tilde u_1^2
+\frac 12 \int \frac{\Phi_{B} P_1'}{\zeta_{B}^2} \tilde u_1^2\\
&\quad+\frac14 \int (\chi_A^2)' (\zeta_B^2)' u_1^2 +\frac12 \int \left[3(\chi_A')^2 + \chi_A'' \chi_A\right]\zeta_B^2 u_1^2\\
&\quad- \int (\chi_A^2)' \Phi_B (\partial_r u_1)^2+\frac14 \int (\chi_A^2)''' \Phi_B u_1^2.
\\
&\quad -\int \left[X_\eps, P_1\right] U w_1\left(\Psi_{A,B} \partial_r u_1+\frac{1}{2}\Psi_{A,B}' u_1\right)\\
&\quad  + \int S_\eps\left[(1+r^2)^{1/4}N\left(\frac{w_1}{(1+r^2)^{1/4}}, r\right)\right] \left(\Psi_{A,B} \partial_r u_1+\frac{1}{2}\Psi_{A,B}' u_1\right)
\\
&=J_1+J_2+J_3+J_4+J_5.
\end{aligned}
\end{equation}
As before we need to estimate consecutive lines. 

Considering $J_1$ we write $J_1=-J_{11}-J_{12}$ where
\[
\begin{aligned}
J_{11}&=\int (\partial_r \tilde u_1)^2+\frac 12 \int \frac{(-\Phi_{B} P_1')}{\zeta_{B}^2} \tilde u_1^2,\\
J_{12}&=\frac{1}{2} \int \left(\frac{\zeta_{B}''}{\zeta_{B}} - \frac{(\zeta_{B}')^2}{\zeta_{B}^2}\right) \tilde u_1^2.
\end{aligned}
\]
We begin with $J_{11}$ which is in some sense the dominant term. For convenience we denote
\[
V_B=-\frac{1}{2} \Phi_{B} P_1'.
\]
and recall that by $V_B\geq 0$ by (\ref{cond:repuls}). 
Note that by Lemma \ref{lem:4.6}
\[
\begin{aligned}
J_{11}&\gtrsim\frac{1}{2} \int (\partial_r \tilde u_1)^2+\frac{1}{2}   J_{11}\\
&\gtrsim \frac{1}{2} \int (\partial_r \tilde u_1)^2+\frac{1}{2} \int (\partial_r \tilde u_1)^2+\frac{1}{2}\int V_B \tilde u_1^2\\
&\gtrsim \frac{1}{2} \int (\partial_r \tilde u_1)^2+\frac{1}{K^2}\int \rho_K^2  \tilde u_1^2 \\
&\gtrsim \frac{1}{4} \int (\partial_r \tilde u_1)^2+\frac{\eps}{K^2} \int \rho_K^2 (\partial_r \tilde u_1)^2+\frac{1}{K^2}\int \rho_K^2  \tilde u_1^2.
\end{aligned}
\]
Using the definition of $\tilde u_1$ we get
\[
\begin{aligned}
\int \rho_K^2  \tilde u_1^2&=\int \rho_K^2 \zeta^2_B u_1^2+\int  \rho_K^2 \zeta^2_B(1-\chi^2_A) u_1^2,\\
\int \rho^2_K(\partial_r \tilde u_1)^2&=\int \rho^2_K\zeta^2_B(\partial_r u_1)^2-\int\rho_K^2\partial_r\left(\zeta_B\chi_A\partial_r(\zeta_B\chi_A)\right) u_1^2\\
&\quad+\int\rho_K^2  \left(\partial_r(\zeta_B\chi_A)\right)^2 u_1^2+\int \rho_K^2\zeta^2_B(\partial_r u_1)^2 (1-\chi_A^2).
\end{aligned}
\]
Let us set 
\[
\begin{aligned}
I_{11}&=\frac{1}{K^2}\int \rho_K^2 \zeta^2_B u_1^2+\frac{\eps}{K^2}\int \zeta^2_B(\partial_r u_1)^2-\int\rho_K^2  \chi^2_A\zeta_B\partial_{rr}\zeta_B u_1^2,\\
I_{12}&=\int  \rho_K^2 \zeta^2_B(1-\chi^2_A) u_1^2-\int\rho_K^2\zeta_B\chi_A\partial_{rr}(\zeta_B\chi_A) u_1^2+\int\rho_K^2  \chi^2_A\zeta_B\partial_{rr}\zeta_B u_1^2+\int \zeta^2_B (1-\chi_A^2)(\partial_r u_1)^2.
\end{aligned}
\]
Now, recall that $u_1=S_\eps w_1$. Note that 
\[
\left|\int\rho_K^2  \chi^2_A\zeta_B\partial_{rr}\zeta_B u_1^2\right|\lesssim \frac{1}{B^2}\int \rho_K^2 \zeta^2_B u_1^2.
\]
From Lemma \ref{lem:transfer} we get as long as $B\gg K$:
\[
I_{11}\gtrsim \frac{1}{K^2} \left(\int \rho_K^2 \sigma^2_B u_1^2+\eps\int \rho_K^2\sigma^2_B(\partial_r u_1)^2\right)\gtrsim\frac{1}{K^3} \|\sigma_B^2 \rho_K^2 w_1\|_{L^2}^2.
\]

Next we will estimate the remaining terms in the definition of  $J_1$. We begin with $J_{12}$. Recalling
\[
\left|\frac{\zeta_{B}''}{\zeta_{B}} - \frac{(\zeta_{B}')^2}{\zeta_{B}^2}\right|\lesssim \frac{1}{B^2} \mathbbm{1}_{[-2,2]},
\]
and using Lemma \ref{lem:coerc} we find
\begin{equation}
\label{est:j12}
|J_{12}| \lesssim \frac{1}{B^2}\int\sigma_B^4\rho_K^4 (S_\eps w_1)^2\lesssim \frac1{B^2\eps} \|\sigma_B^2\rho_K^2 w_1\|_{L^2}^2.
\end{equation}
We consider now $I_{12}$. We have
\[
\begin{aligned}
\left|\int  \rho_K^2 \zeta^2_B(1-\chi^2_A) u_1^2\right| & \lesssim e^{\,-A/K} \int \sigma_A^2 (S_\eps w_1)^2\lesssim \eps^{-1} e^{\,-A/K} \|\sigma_A w_1\|^2_{L^2},\\
\left|\int \rho_K^2\zeta_B\chi_A\partial_r\zeta_B\partial_r\chi_A u_1^2\right| & \lesssim e^{\,-A/B} \int \sigma_B u_1^2\lesssim e^{\,-A/B} \int \sigma_A^2 (S_\eps w_1)^2\lesssim \eps^{-1} e^{\,-A/B}\|\sigma_A w_1\|^2_{L^2},
\\
\left|\int \rho_K^2\zeta_B^2 \chi_A \partial_{rr}\chi_A u_1^2\right| & \lesssim e^{\,-A/B} \int \sigma_B u_1^2\lesssim \eps^{-1} e^{\,-A/B} \|\sigma_A w_1\|^2_{L^2},\\
\left|\int \rho_K^2 \zeta^2_B (1-\chi_A^2)(\partial_r u_1)^2\right| &\lesssim e^{\,-A/B} \int \sigma_B (\partial_r u_1)^2\lesssim \eps^{-2} e^{\,-A/B} \|\sigma_A w_1\|^2_{L^2}.
\end{aligned}
\]
It follows, when $K\ll B\ll A$,
\begin{equation}
\label{est:jone}
-J_1=J_{11}+J_{12}\gtrsim \frac{1}{4} \int \sigma^2_B(\partial_r u_1)^2+\left(K^{-3}-\eps^{-1}B^{-2}\right)\|\sigma_B^2 \rho_K^2 w_1\|_{L^2}^2-\eps^{-2} e^{\,-A/B}\|\sigma_A w_1\|^2_{L^2}.
\end{equation}

To estimate of $J_2$ we follow similar lines as in the bound  for $I_{12}$ to obtain
\begin{equation}
\label{est:jtwo}
|J_2|\lesssim \eps^{-1} e^{\,-A/B}\|\sigma_A w_1\|^2_{L^2}.
\end{equation}

Now we consider $J_3$. We have
\[
\left| \int (\chi_A^2)' \Phi_B (\partial_r u_1)^2\right|\lesssim \frac{B}{A} \int \sigma^2_A (\partial_r S_\eps w_1)^2.
\]
We write
\[
\partial_r S_\eps w_1=\partial_r X_\eps U w_1=\partial_r X_\eps \partial_r w_1+X_\eps \partial_r(\nu w_1)=X_\eps \partial_{rr} w_1+X_\eps \partial_r(\nu w_1).
\]
Using  Lemma \ref{lem:comut 1} with $g=\sigma_A$ we obtain
\begin{equation}
\label{est:sax one}
\begin{aligned}
\|\sigma_A X_\eps \partial_{rr} w_1\|_{L^2}^2 & \lesssim \|X_\eps(\sigma_A \partial_{rr} w_1)\|_{L^2}^2
\\
& \lesssim  \|X_\eps\partial_r(\sigma_A \partial_{r} w_1)\|_{L^2}^2+\|X_\eps(\partial_r\sigma_A \partial_{r} w_1)\|_{L^2}^2
\\
&\lesssim \eps^{-1} \|\sigma_A \partial_{r} w_1\|_{L^2}^2.
\end{aligned}
\end{equation}
We also get
\begin{equation}
\label{est:sax two}
\|\sigma_A X_\eps \partial_r(\nu w_1)\|^2_{L^2}\lesssim \|\sigma_A \partial_r(\nu w_1)\|^2_{L^2}\lesssim \|\sigma_A \partial_r\nu w_1\|^2_{L^2}+ \|\sigma_A \nu \partial_r w_1\|^2_{L^2}.
\end{equation}
The most complicated term is 
\begin{equation}
\label{est:sax threee}
\|\sigma_A \partial_r\nu w_1\|^2_{L^2}=\int \sigma^2_A (\partial_ r\nu)^2 w_1^2.
\end{equation}
We have
\[
(\partial_r\nu)^2=\frac{9}{4}\frac{(1-r^2)^2}{(1+r^2)^4}\lesssim \frac{\Phi_A}{r} \frac{1}{1+r^2},
\]
hence
\begin{equation}
\label{est:sax three}
\|\sigma_A \partial_r\nu w_1\|^2_{L^2}\lesssim \int \frac{\Phi_A}{r} \frac{1}{1+r^2} w_1^2.
\end{equation}
We also have
\[
\left|\int (\chi_A^2)''' \Phi_B u_1^2\right|\lesssim \frac{B}{A^3}\|\sigma_A S_\eps w_1\|^2_{L^2}\lesssim \frac{B}{\eps A^3}\|\sigma_Aw_1\|^2_{L^2}.
\]
Combing the above estimates we get
\begin{equation}
\label{est:jthree}
\begin{aligned}
|J_3|\lesssim \left| \int (\chi_A^2)' \Phi_B (\partial_r u_1)^2\right|+\left|\int (\chi_A^2)''' \Phi_B u_1^2\right|&\lesssim \frac{B}{\eps A} \|\sigma_A \partial_{r} w_1\|_{L^2}^2+\frac{B}{A} \int \frac{\Phi_A}{r} \frac{1}{1+r^2} w_1^2\\
&\quad +\frac{B}{\eps A^3}\|\sigma_Aw_1\|^2_{L^2}.
\end{aligned}
\end{equation}
Now will estimate $J_4$. We write 
\[
J_4=J_{41}+J_{42},
\]
where
\[
\begin{aligned}
J_{41}&=\int \left(\left[X_\eps, P_1\right] U w_1\right)\Psi_{A,B} \partial_r u_1,
\\
J_{42}&= \int \left(\left[X_\eps, P_1\right] \right)U w_1 \Psi_{A,B}' u_1.
\end{aligned}
\]
It is not hard to show that 
\begin{equation}
\label{eq:comut 1}
\left[X_\eps, P_1\right] U w_1=2\eps X_\eps(\partial_r P_1 \partial_r S_\eps w_1)+\eps X_\eps(\partial_{rr} P_1 S_\eps w_1).
\end{equation}
Recalling that $u_1=S_\eps w_1$ we have
\[
J_{41}=2\eps \int X_\eps(\partial_r P_1 \partial_r u_1)  \Psi_{A,B} \partial_r u_1+\eps\int X_\eps(\partial_{rr} P_1 u_1) \Psi_{A,B} \partial_r u_1=2\eps I_{41}+\eps I_{42}.
\]
To estimate $I_{41}$ we decompose
\[
I_{41}=\int X_\eps(\partial_r P_1 \partial_ru_1)  \Psi_{A,B}\chi_B^2 \partial_r u_1+\int X_\eps(\partial_r P_1 \partial_ru_1)  \Psi_{A,B}(1-\chi_B^2) \partial_r u_1.
\]
We have
\[
\begin{aligned}
\left|\int X_\eps(\partial_r P_1 u_1)  \Psi_{A,B}\chi_B^2  \partial_ru_1\right| & \lesssim B\|\chi_B X_\eps(\partial_r P_1 \partial_ru_1)\|_{L^2}\|\chi_B \partial_r u_1\|_{L^2}\\
&\lesssim B\|\sigma_B X_\eps(\partial_r P_1\partial_r u_1)\|_{L^2}\|\sigma_B \partial_r u_1\|_{L^2}\\
&\lesssim B\|\sigma_B \partial_r P_1 \partial_ru_1\|_{L^2}\|\sigma_B \partial_r u_1\|_{L^2}\\
&\lesssim B\|\sigma_B \partial_r u_1\|^2_{L^2}.
\end{aligned}
\]
To estimate the second term in the decomposition  of $I_{41}$ we define 
\[
\rho(r)=(1+r^2)^{-\frac{3}{2}},
\]
and write the following bound
\begin{equation}
\label{est:i41two}
\begin{aligned}
\left|\int X_\eps(\partial_r P_1 u_1)  \Psi_{A,B}(1-\chi_B^2) \partial_r u_1 \right| & \lesssim  \|(1-\chi_B^2)^{\frac{1}{2}}\Psi^{\frac{1}{2}}_{A,B}\rho^{-1} X_\eps (\partial_r P_1 \partial_r u_1)\|_{L^2} \\
&\quad \times \|(1-\chi_B^2)^{\frac{1}{2}}\Psi^{\frac{1}{2}}_{A,B}\rho \partial_r u_1\|_{L^2}.
\end{aligned}
\end{equation}
We have, arguing similarly as in  (\ref{est:sax one})--(\ref{est:sax three})
\[
\begin{aligned}
\|(1-\chi_B^2)^{\frac{1}{2}}\Psi^{\frac{1}{2}}_{A,B}\rho \partial_r u_1\|^2_{L^2}&\lesssim \frac{1}{B^2} \|\sigma_A \partial_r S_\eps w_1\|^2_{L^2}\\
&\lesssim \frac{1}{\eps B^2} \|\sigma_A \partial_{r} w_1\|_{L^2}^2+\frac{1}{B^2} \int \frac{\Phi_A}{r} \frac{1}{1+r^2} w_1^2.
\end{aligned}
\]
As for the remaining term on the right hand side of (\ref{est:i41two}) we have using Lemma \ref{lem:comut 1}
\[
\begin{aligned}
\|(1-\chi_B^2)^{\frac{1}{2}}\Psi^{\frac{1}{2}}_{A,B}\rho^{-1} X_\eps (\partial_r P_1 \partial_r u_1)\|^2_{L^2} &\lesssim B\|\sigma_A\rho^{-1} X_\eps(\partial_r P_1 \partial_r u_1)\|^2_{L^2}
\\
&\lesssim B \|X_\eps (\sigma_A \rho^{-1} \partial_r P_1 \partial_r u_1)\|^2_{L^2}\\
&\lesssim B \|\sigma_A \rho^{-1} \partial_r P_1 \partial_r u_1\|^2_{L^2}\\
&\lesssim \frac{1}{B^2} \|\sigma_A \partial_r S_\eps w_1\|^2_{L^2}\\
&\lesssim \frac{1}{\eps B^2} \|\sigma_A \partial_{r} w_1\|_{L^2}^2+\frac{1}{B^2} \int \frac{\Phi_A}{r} \frac{1}{1+r^2} w_1^2.
\end{aligned}
\]
Summarizing we have
\begin{equation}
\label{est: i41}
\begin{aligned}
|I_{41}|&\lesssim B\|\sigma_B \partial_r u_1\|^2_{L^2}+\frac{1}{\eps B^2} \|\sigma_A \partial_{r} w_1\|_{L^2}^2+\frac{1}{B^2} \int \frac{\Phi_A}{r} \frac{1}{1+r^2} w_1^2.
\end{aligned}
\end{equation}

Similarly we write
\[
|I_{42}|=\left|\int X_\eps(\partial_{rr} P_1 u_1)\rho^{-1} \Psi_{A,B}\rho \partial_r u_1\right|\lesssim  \|X_\eps(\partial_{rr}P_1 u_1) \Psi_{A,B}^{\frac{1}{2}} \rho^{-1}\|_{L^2}\|\partial_r u_1 \Psi_{A,B}^{\frac{1}{2}} \rho\|_{L^2}.
\]
We have
\[
\|\partial_r u_1 \Psi_{A,B}^{\frac{1}{2}} \rho\|_{L^2}\lesssim \|\partial_r u_1 \Psi_{A,B}^{\frac{1}{2}}\rho\sigma_B\|_{L^2}+\|\partial_r u_1 \Psi_{A,B}^{\frac{1}{2}} \rho(1-\chi_B^2)^{\frac{1}{2}}\|_{L^2}.
\]
The first term on the right is easy to control
\[
\|\partial_r u_1 \Psi_{A,B}^{\frac{1}{2}}\rho\sigma_B\|_{L^2}\lesssim  B^{\frac{1}{2}}\|\partial_r u_1 \sigma_B\|_{L^2}.
\]
The second term is estimated similarly as in (\ref{est:sax one})--(\ref{est:sax three})
\[
\begin{aligned}
\|\partial_r u_1 \Psi_{A,B}^{\frac{1}{2}} \rho(1-\chi_B^2)^{\frac{1}{2}}\|_{L^2}&\lesssim \frac{1}{B} \|\partial_r u_1\sigma_A\|_{L^2}\\
&\lesssim \frac{1}{B\eps^{\frac{1}{2}}} \|\partial_r w_1\sigma_A\|_{L^2}+\frac{1}{B}\left(\int \frac{\Phi_A}{r}\frac{1}{1+r^2} w_1^2\right)^{\frac{1}{2}}.
\end{aligned}
\]
Next we use Lemma \ref{lem:comut 1} to estimate
\[
\|X_\eps(\partial_{rr} P_1 u_1) \Psi_{A,B}^{\frac{1}{2}} \rho^{-1}\|_{L^2}\lesssim B^{\frac{1}{2}}\|X_\eps(\partial_{rr} P_1 u_1) \sigma_A \rho^{-1}\|_{L^2}\lesssim B^{\frac{1}{2}}\|u_1\partial_{rr} P_1   \rho^{-1} \sigma_A\|_{L^2}.
\]
We note that 
\[
|\partial_{rr} P_1|\rho^{-1}\lesssim \frac{1}{(1+r^2)^2}\rho^{-1}\lesssim \frac{1}{(1+r^2)^{\frac{1}{2}}}.
\] 
Recall that 
\[
u_1=S_\eps w_1=X_\eps \partial_r w_1+X_\eps(\nu w_1).
\]
Then 
\[
\begin{aligned}
\|u_1\partial_{rr} P_1   \rho^{-1} \sigma_A\|_{L^2}&\lesssim \|X_\eps(\partial_r w_1) (1+r^2)^{-\frac{1}{2}} \sigma_A\|_{L^2}+\|X_\eps(\nu w_1) (1+r^2)^{-\frac{1}{2}} \sigma_A\|_{L^2}\\
&\lesssim \|\partial_r w_1 \sigma_A\|_{L^2}+\| w_1 \nu(1+r^2)^{-\frac{1}{2}} \sigma_A\|_{L^2}\\
&\lesssim \|\partial_r w_1 \sigma_A\|_{L^2}+\left(\int \frac{\Phi_A}{r}\frac{1}{1+r^2} w_1^2\right)^{\frac{1}{2}}.
\end{aligned}
\]
Summarizing,
\begin{equation}
\begin{aligned}
\label{est: i42}
|I_{42}|&\lesssim B\|\partial_r u_1 \sigma_B\|^2_{L^2}+\left(B+\frac{1}{B^2\eps}\right) \|\partial_r w_1\sigma_A\|_{L^2}+\left(B+\frac{1}{B^2}\right)\int \frac{\Phi_A}{r}\frac{1}{1+r^2} w_1^2.
\end{aligned}
\end{equation}

We  turn our attention to the term $J_{42}$. Using (\ref{eq:comut 1}) we write
\[
J_{42}=2\eps I_{43}+\eps I_{44},
\]
where
\[
\begin{aligned}
I_{43}&= \int X_\eps(\partial_r P_1 \partial_r u_1)  \Psi_{A,B}' u_1,
\\
I_{44}&=\int X_\eps(\partial_{rr} P_1 u_1) \Psi_{A,B}' u_1.
\end{aligned}
\]
Recalling the definition of $\Psi_{A,B}$ we have
\[
\int X_\eps(\partial_r P_1 \partial_r u_1)  \Psi_{A,B}' u_1=\frac{2}{A} \int X_\eps(\partial_r P_1 \partial_r u_1)\chi_A \chi'_A\Phi_B u_1+\int X_\eps(\partial_r P_1 \partial_r u_1)\chi^2_A\zeta^2_B u_1.
\]
Consider first
\[
\begin{aligned}
\left| \int X_\eps(\partial_r P_1 \partial_r u_1)\chi_A \chi'_A\Phi_B u_1\right|&\lesssim  B \int |X_\eps(\partial_r P_1 \partial_r u_1)|\rho^{-1}\sigma^2_A  |u_1|\rho
\\
&\lesssim B\left(\int |X_\eps(\partial_r P_1 \partial_r u_1)|^2\rho^{-2}\sigma^2_A\right)^{\frac{1}{2}} \left(\int |u_1|^2\rho^2 \sigma^2_A \right)^{\frac{1}{2}}.
\end{aligned}
\]
Using Lemma \ref{lem:comut 1} we have
\[
\begin{aligned}
\|X_\eps(\partial_r P_1 \partial_r u_1)\rho^{-1}\sigma_A\|_{L^2}^2 & \lesssim \|\partial_r P_1 \partial_r u_1\rho^{-1}\sigma_A\|_{L^2}^2
\\
&\lesssim \|\partial_r u_1 \sigma_A\|^2_{L^2}
\\
&\lesssim \frac{1}{\eps} \|\sigma_A \partial_{r} w_1\|_{L^2}^2+ \int \frac{\Phi_A}{r} \frac{1}{1+r^2} w_1^2.
\end{aligned}
\]
again arguing similarly as in  (\ref{est:sax one})--(\ref{est:sax three}) in the last estimate. 
We also have
\[
\begin{aligned}
\|u_1 \rho\sigma_A\|_{L^2}^2 &\lesssim \|X_\eps(\partial_r w_1) \rho\sigma_A\|^2_{L^2}+\| X_\eps(\nu w_1)\rho\sigma_A\|^2_{L^2}\\
&\lesssim \|\sigma_A \partial_{r} w_1)\|_{L^2}^2+ \int \frac{\Phi_A}{r} \frac{1}{1+r^2} w_1^2.
\end{aligned}
\]
Next we consider
\[
\begin{aligned}
\left|\int X_\eps(\partial_r P_1 \partial_r u_1)\chi^2_A\zeta^2_B u_1\right|&\lesssim \|X_\eps(\partial_r P_1 \partial_r u_1)\rho^{-1}\zeta_B\|_{L^2} \|u_1\rho\zeta_B\|_{L^2}\\
&\lesssim \|\partial_r P_1\partial_r u_1 \rho^{-1} \zeta_B\|^2_{L^2} + \|u_1\rho \sigma_B\|^2_{L^2}\\
&\lesssim \|\partial_r u_1\sigma_B\|^2_{L^2}+\|\sigma_A \partial_{r} w_1\|_{L^2}^2+ \int \frac{\Phi_A}{r} \frac{1}{1+r^2} w_1^2.
\end{aligned}
\]
Thus
\[
\begin{aligned}
|I_{43}|&\lesssim \left(\frac{B}{A\eps}+1\right) \|\sigma_A \partial_{r} w_1\|_{L^2}^2+ \left(\frac{B}{A}+1\right)\int \frac{\Phi_A}{r} \frac{1}{1+r^2} w_1^2+ \|\partial_r u_1\sigma_B\|^2_{L^2}.
\end{aligned}
\]
To estimate $I_{44}$ we argue similarly to get
\[
\begin{aligned}
|I_{44}|&\lesssim \left(\frac{B}{A}+1\right) \|\sigma_A \partial_{r} w_1\|_{L^2}^2+ \left(\frac{B}{A}+1\right)\int \frac{\Phi_A}{r} \frac{1}{1+r^2} w_1^2.
\end{aligned}
\]
Summarizing all this we get
\begin{equation}
\begin{aligned}
\label{est:jfour}
|J_4|&\lesssim (\eps+\eps B)\|\sigma_B \partial_r u_1\|^2_{L^2}+\left(\eps+\frac{B}{A}+\frac{1}{B^2}+\eps B\right) \|\sigma_A \partial_{r} w_1\|_{L^2}^2\\
&\quad +\left(\eps+\frac{B}{A}+\eps B+\frac{\eps}{B^2}\right) \int \frac{\Phi_A}{r} \frac{1}{1+r^2} w_1^2.
\end{aligned}
\end{equation}

It remains to estimate 
\[
J_5=\int S_\eps\left[(1+r^2)^{1/4}N\left(\frac{w_1}{(1+r^2)^{1/4}}, r\right)\right] \left(\Psi_{A,B} \partial_r u_1+\frac{1}{2}\Psi_{A,B}' u_1\right).
\]
We decompose
\[
\begin{aligned}
J_{51}&=\int S_\eps\left[(1+r^2)^{1/4}N\left(\frac{w_1}{(1+r^2)^{1/4}}, r\right)\right] \Psi_{A,B} \partial_r u_1,\\
J_{52}&=\frac{1}{2}\int S_\eps\left[(1+r^2)^{1/4}N\left(\frac{w_1}{(1+r^2)^{1/4}}, r\right)\right] \Psi_{A,B}' u_1.
\end{aligned}
\]
Using the explicit expression for $N$ we write $J_{51}=I_{51}+I_{52}$
\[
\begin{aligned}
I_{51}&=\int S_\eps\left(\frac {-6 H w_1^2}{(1+r^2)^{5/4}}\right) \Psi_{A,B}\partial_r u_1,\\
I_{52}&=\int S_\eps\left(\frac {-2 w_1^3}{(1+r^2)^{3/2}}\right) \Psi_{A,B}\partial_r u_1.
\end{aligned}
\] 
Recall that
\[
\|\partial_r u_1\sigma_A\|_{L^2}^2 \lesssim \left(\eps^{-1} \|\partial_r w_1\sigma_A\|_{L^2}^2+ \int\frac{\Phi_A}{r}\frac{1}{1+r^2} w_1^2\right),
\]
also,
\[
\begin{aligned}
\|S_\eps (w_1^2  H(1+r^2)^{-5/4})\sigma_A\|_{L^2} &\lesssim \| \sigma_A X_\eps (\partial_r (w_1^2 H (1+r^2)^{-5/4}))\|_{L^2}\\
&\quad +\|\sigma_A X_\eps (\nu w_1^2 H(1+r^2)^{-5/4})\|_{L^2}.
\end{aligned}
\]
Using Lemma \ref{lem:xeps} and Lemma \ref{lem:comut 1} we have
\begin{align*}
\| \sigma_A X_\eps (\partial_r (w_1^2 H (1+r^2)^{-5/4}))\|_{L^2}&\lesssim \|\sigma_A \partial_r w_1 w_1 H\ (1+r^2)^{-5/4}\|_{L^2}+\|\sigma_A w_1^2 (1+r^2)^{-11/4}\|_{L^2}\\
&\quad +\|\sigma_A w_1^2 (1+r^2)^{-9/4}\|_{L^2}\\
&\lesssim \norm{\frac{w_1}{(1+r^2)^{1/4}}}_{L^\infty} \| \sigma_A \partial_r w_1\|_{L^2}\\
&\quad + \norm{\frac{w_1}{(1+r^2)^{1/4}}}_{L^\infty}\left(\int\frac{\Phi_A}{r} \frac{w_1^2}{(1+r^2)} \right)^{1/2}.
\end{align*}
On the other hand,
\begin{align*}
\|\sigma_A X_\eps (\nu w_1^2 H(1+r^2)^{-5/4})\|_{L^2} &\lesssim \| \sigma_A \nu w_1^2 H (1+r^2)^{-5/4}\|_{L^2}\\
&\lesssim \norm{\sigma_A \nu \frac{w_1^2}{(1+r^2)^{1/2}}}_{L^2}\\
&\lesssim \norm{\frac{w_1}{(1+r^2)^{1/4}}}_{L^\infty}  \|\sigma_A w_1 \nu\|_{L^2}\\
&\lesssim \norm{\frac{w_1}{(1+r^2)^{1/4}}}_{L^\infty} \left( \int \frac{\Phi_A}{r}\frac{w_1^2}{1+r^2} \right)^{1/2}.
\end{align*}
From the previous estimates:
\begin{align*}
\|S_\eps (w_1^2  H(1+r^2)^{-5/4})\sigma_A\|_{L^2}&\lesssim \norm{\frac{w_1}{(1+r^2)^{1/4}}}_{L^\infty} \left(\|\sigma_A \partial_r w_1\|_{L^2}+\left(\int\frac{\Phi_A}{r}\frac{1}{1+r^2} w_1^2 \right)^{1/2} \right)\\
&\lesssim \norm{\frac{w_1}{(1+r^2)^{1/4}}}_{L^\infty}\left(\|\sigma_A \partial_r w_1\|^2 +\int\frac{\Phi_A}{r}\frac{1}{1+r^2} w_1^2 \right)^{1/2}.
\end{align*}
Combining the above  we get:
\begin{equation}
\label{est:i51}
\begin{aligned}
|I_{51}|&\lesssim B\|S_\eps (w_1^2  H(1+r^2)^{-5/4})\sigma_A\|_{L^2}\|\partial_r u_1\sigma_A\|_{L^2}
\\
&\lesssim  B\norm{\frac{w_1}{(1+r^2)^{1/4}}}_{L^\infty} \left(\eps^{-1} \|\partial_r w_1\sigma_A\|_{L^2}^2+ \int\frac{\Phi_A}{r}\frac{1}{1+r^2} w_1^2\right).
\end{aligned}
\end{equation}

Similarly we estimate $I_{52}$:
\begin{equation}
\label{est:i52}
|I_{52}| \lesssim B\norm{\frac{w_1}{(1+r^2)^{1/4}}}_{L^\infty} \left(\eps^{-1} \|\partial_r w_1\sigma_A\|_{L^2}^2+ \int\frac{\Phi_A}{r}\frac{1}{1+r^2} w_1^2\right).
\end{equation}
Following similar steps as above it is now not hard to prove
\begin{equation}
\label{est:j52}
|J_{52}|\lesssim B\norm{\frac{w_1}{(1+r^2)^{1/4}}}_{L^\infty} \left(\|\partial_r w_1\sigma_A\|_{L^2}^2+ \int\frac{\Phi_A}{r}\frac{1}{1+r^2} w_1^2\right).
\end{equation}
Combining (\ref{est:i51})--(\ref{est:j52}) we obtain
\begin{equation}
\label{est:j5}
|J_5|\lesssim B\norm{\frac{w_1}{(1+r^2)^{1/4}}}_{L^\infty} \left(\eps^{-1} \|\partial_r w_1\sigma_A\|_{L^2}^2+ \int\frac{\Phi_A}{r}\frac{1}{1+r^2} w_1^2\right).
\end{equation}

\section{Proof of Theorem \ref{thm:2}}
\setcounter{equation}{0}
We summarize estimates (\ref{est:jone}), (\ref{est:jtwo}), (\ref{est:jthree}), (\ref{est:jfour}), (\ref{est:j5}):
\begin{equation}
\label{est:jall}
\begin{aligned}
-\dot{\mathcal J}&\gtrsim \left(\frac{1}{4} -\eps-\eps B\right)\int \sigma^2_B(\partial_r u_1)^2+\left(K^{-3}-\eps^{-1}B^{-2}\right)\|\sigma_B^2 \rho_K^2 w_1\|_{L^2}^2\\
&\quad -\left(B\delta\eps^{-1}+\eps+\eps B+\frac{1}{B^2}+\frac{B}{\eps A}+\eps^{-2} e^{\,-A/B}\right)\|\sigma_A \partial_r w_1\|^2_{L^2}\\
&\quad -\frac{B}{\eps A^3}\|\sigma_Aw_1\|^2_{L^2}\\
-&\quad\left(B\delta+ \eps+\eps B+\frac{\eps}{B^2}+  \frac{B}{A}   \right) \int\frac{\Phi_A}{r}\frac{1}{1+r^2} w_1^2,
\end{aligned}
\end{equation}
and note that (\ref{est:virial12}) implies:
\begin{equation}
\label{est:ih}
-\dot{\mathcal I}+\frac{1}{A^2}  \dot{\mathcal H}+ C_8\|\sigma_B^2\rho_K^2 w_1\|^2_{L^2}\gtrsim \int \sigma^2_A |\partial_r w_1|^2+\frac{1}{A^2}\int \sigma_A^2 w_1^2+\frac{1}{A^2}\int \sigma_A^2 w_2^2+\int  \frac{\Phi_A}{r} \frac{1}{1+r^2} w_1^2.
\end{equation}
Next, we choose parameters $K, B, \eps, A, \delta$. First we choose $K$ as in (\ref{eq:choice k}). We will need $K\ll B\ll A$. Next we pick $B$ large and $\epsilon$ small such that $\eps B\ll 1$ and 
$\eps^{-1} B^{-2}\ll K^{-3}$. Then we chose $A$ such that $B\eps^{-1} A^{-1}\ll 1$  and $\eps^{-2} e^{\,-A/B}\ll 1$. Finally we chose $\delta$ small such that $B\delta\eps^{-1}\ll 1$ and  $A^2\delta \ll A^{-2}$. With these choices there exists $\eta\ll 1$ independent on $K$ such that 
\[
-\dot{\mathcal J}-\eta\left(-\dot{\mathcal I}+\frac{1}{A^2}  \dot{\mathcal H}\right)\gtrsim K^{-3} \|\sigma_B^2 \rho_K^2 w_1\|_{L^2}^2.
\]
Integrating the above in $t$ and using the fact that 
\[
|\mathcal I(t)|, |\mathcal J(t)|, |\mathcal H(t)|\lesssim A \delta^2,
\]
we infer that for each $T>0$ 
\begin{equation}
\label{eq:intone}
\int_0^T \|\sigma_B^2 \rho_K^2 w_1\|_{L^2}^2\lesssim A\delta^2.
\end{equation}
Integrating now (\ref{est:ih}) in $t$ we get
\begin{equation}
\label{eq:inttwo}
\int_0^T \|\sigma_A \partial_r w_1\|_{L^2}^2+\frac{1}{A^2}\int_0^T \|\sigma_A w_1\|_{L^2}^2+\frac{1}{A^2}\int_0^T \|\sigma_A w_2\|_{L^2}^2+\int_0^T  \int \frac{\Phi_A}{r} \frac{1}{1+r^2} w_1^2\lesssim A\delta^2. 
\end{equation}
The assertion of the Theorem \ref{thm:2} follows now from (\ref{eq:intone}) and (\ref{eq:inttwo}) by a standard argument (c.f. \cite{KM1}). 
\normalsize

\end{document}